\documentclass[11pt]{article}
\usepackage{amssymb,amsmath,amsthm, amscd,float}
\usepackage{graphicx, color,enumerate}
\usepackage{hyperref}

\theoremstyle{plain}
\newtheorem{theorem}{Theorem}

\newtheorem*{MRT}{Main Theorem}
\newtheorem{lemma}[theorem]{Lemma}

\newtheorem{claim}[theorem]{Claim}
\newtheorem{cor}[theorem]{Corollary}

\newtheorem{prop}[theorem]{Proposition}
\theoremstyle{definition}
\newtheorem{defn}[theorem]{Definition}
\newtheorem{definition}[theorem]{Definition}
\newtheorem{example}[theorem]{Example}
\newtheorem{remark}[theorem]{Remark}

\def\ZZ {{\mathbb Z}}

\def\RR {{\mathbb R}}

\def\R{{\mathbb R}}
\def\Z{{\mathbb Z}}
\newcommand{\lam}{\lambda}

\newcommand{\al}{\alpha}

\newcommand{\eps}{\varepsilon}
\newcommand{\out}{\textup{Out}(F_n)}

\newcommand{\X}{\mathcal{X}}

\newcommand{\Sig}{\Sigma}

\renewcommand{\path}{\mathfrak{p}}

\title{Asymmetry of
  Outer Space} 
\author{Yael Algom-Kfir and Mladen Bestvina\thanks{The
    second author gratefully acknowledges the support by the National
    Science Foundation. }}
\date{December 22, 2010}
\begin{document}
\maketitle

\begin{abstract}
We study the asymmetry of the Lipschitz metric $d$ on Outer space. We
introduce an (asymmetric) Finsler norm $\|\cdot\|^L$ that induces $d$. There
is an $Out(F_n)$-invariant ``potential'' $\Psi$ defined on Outer space such
that when $\|\cdot\|^L$ is corrected by $d\Psi$, the resulting norm is
quasi-symmetric. As an application, we give new proofs of two theorems
of Handel-Mosher, that $d$ is quasi-symmetric when
restricted to a thick part of Outer space, and that there is a uniform
bound, depending only on the rank, on the ratio of logs of growth
rates of any irreducible $f\in Out(F_n)$ and its inverse.
\end{abstract}

\section{Introduction}

Teichm\"uller space can be equipped with three natural metrics:
the Teichm\"uller metric, the Weil-Petersson metric and Thurston's Lipschitz
metric \cite{thurston}. It is only the latter one that has an analog
in Outer space. The first systematic study of the Lipschitz metric in
Outer space was conducted by Francaviglia-Martino \cite{FM}. Just like
Thurston's metric, this metric is not symmetric, but it does have many
useful properties. For example, it is a geodesic metric, $Out(F_n)$
acts on Outer space by isometries, and if $\Phi\in Out(F_n)$ is a
fully irreducible automorphism then the
translation distance of $\Phi$ equals $\log\lambda$,
the growth rate of $\Phi$. 
This last property was exploited to give a new
proof in \cite{mb} of the train track theorem \cite{BH}.
Moreover, $\Phi$ acts as a translation by
$\log\lambda$ on certain biinfinite geodesics, called {\it axes}. In
\cite{yael} it is shown that axes are {\it strongly contracting},
pointing to negative curvature properties of the Lipschitz metric in
these directions.

In this paper we introduce an asymmetric Finsler norm on the tangent
vectors of Outer space that induces the Lipschitz metric. We also show
how to correct this norm to make it quasi-symmetric. Our main result
explains the lack of quasi-symmetry in terms of a certain potential
function.

\begin{MRT}
There is an $Out(F_n)$-invariant continuous, piecewise analytic
function $\Psi:\X_n\to\R$ and constants $A,B>0$ (depending only on
$n$) such that for every $x,y\in\X_n$ we have
$$d(y,x)\leq A\ d(x,y)+B[\Psi(y)-\Psi(x)]$$
\end{MRT}

As an application we get a new proof (Theorem \ref{HMTheorem}) of
Handel and Mosher's result \cite{HM} that the expansion factor of an
irreducible automorphism is bounded by a power of the expansion factor
of the inverse automorphism (it is well known that in general they
need not be equal). We also get an easy proof that in the subspace of
Outer Space of the points whose underlying graph has injectivity
radius bounded from below, the Lipschitz metric is symmetric up to a
multiplicative error (Theorem \ref{sym_in_thick}).

\noindent
{\bf Acknowledgements.} We thank Bert Wiest and the referee for
helpful comments.

\subsection{Outer space and tangent spaces}
A {\it graph} will always
be a finite cell complex of dimension 1 with all vertices of valence
$>2$. A {\it
  metric} on a graph $\Gamma$ is a function $\ell:E(\Gamma)\to [0,1]$
defined on the set of edges of $\Gamma$
such that
\begin{itemize}
\item $\sum_{e\in E(\Gamma)}\ell(e)=1$, and
\item $\cup_{\ell(e)=0}e$ is a forest, i.e. it contains no circles.
\end{itemize}

The space $\Sigma_\Gamma$ of all metrics $\ell$ on $\Gamma$ is a
``simplex with missing faces''; the missing faces correspond to
degenerate metrics that vanish on a subgraph which is not a forest. 

When $\ell\in\Sigma_\Gamma$, we have the tangent space
$$T_\ell(\Sigma_\Gamma)=\{\tau:E(\Gamma)\to\R\mid \sum_{e\in E(\Gamma)}\tau(e)=0\}$$
If $\ell,\ell'$ are two points in $\Sigma_\Gamma$ the natural
identification between $T_\ell(\Sigma_\Gamma)$ and
$T_{\ell'}(\Sigma_\Gamma)$ leads to a product decomposition
$$T(\Sigma_\Gamma)\cong \Sigma_\Gamma\times\R^{N-1}$$ of the total
tangent space, 
where $N$ is the number of edges of $\Gamma$.

A tangent vector $\tau\in T_\ell(\Sigma_\Gamma)$ is {\it integrable}
if $\tau(e)<0$ implies $\ell(e)>0$ for all $e\in E(\Gamma)$. In that
case we have the path $\ell+t\tau\in \Sigma_\Gamma$ for small $t\geq
0$. 

If $\Gamma'$ is obtained from $\Gamma$ by collapsing a forest, then we
have natural inclusions $\Sigma_{\Gamma'}\subset \Sigma_\Gamma$ and
$T(\Sigma_{\Gamma'})\subset T(\Sigma_\Gamma)$ given by considering
metrics on $\Gamma$ that vanish on the forest.

Let $F_n$ denote the free group of rank $n$. 
The rose $R_n$ is the wedge of $n$ circles. A {\it marking} is a
homotopy equivalence $f:R_n\to\Gamma$ from the rose to a graph. A {\it
  marked graph} is a pair $(\Gamma,f)$ where $f:R_n\to\Gamma$ is a
marking. Two marked graphs $(\Gamma,f)$ and $(\Gamma',f')$ are {\it
  equivalent} if there is a homeomorphism $\phi:\Gamma\to\Gamma'$ so
that $\phi f\simeq f'$ (homotopic).

Recall \cite{CV} that
Culler-Vogtmann's {\it Outer space} $\X_n$ is obtained from the
disjoint union 
$$\amalg_{(\Gamma,f)}\Sigma_{\Gamma}$$ by identifying the faces of the simplices along the above
inclusions, where the union is taken over the representatives of
equivalence classes of marked graphs $(\Gamma,f)$. Thus a point of
$\X_n$ is represented by a triple $(\Gamma,f,\ell)$, and we will
usually blur the distinction between such a triple and the equivalence
class it represents. If $\alpha$ is an immersed curve in $\Gamma$ we
define $\ell(\alpha)$ as the sum of the lengths of edges $\alpha$
crosses, with multiplicity. If $\alpha$ is not immersed we first
tighten to an immersed loop and then compute the length. Similarly, if $\tau \in \Sigma_\Gamma$ and $\al$ a loop in $\Gamma$ then $\tau(\alpha)$ is  the sum of weights on the edges which are crossed by the immersed loop that is freely homotopic to $\alpha$. 

The outer
automorphism group $Out(F_n)$ acts on $\X_n$ on the right by
precomposing:
$$(\Gamma,f,\ell)\cdot \Phi=(\Gamma,f\Phi,\ell)$$
where the group of homotopy equivalences (up to homotopy) of $R_n$ is
identified with $Out(F_n)$.

\subsection{Lipschitz metric on $\X_n$}

Let $(\Gamma,f,\ell),(\Gamma',f',\ell')$ represent two points $x,y$ in
$\X_n$. A {\it difference of markings} is a map
$\phi:\Gamma\to\Gamma'$ with $\phi f\simeq f'$. We will always assume
that $\phi$ is linear on edges. By $\sigma(\phi)$ denote the largest
slope of $\phi$ over all edges of $\Gamma$.  Define the distance
$$d(x,y)=\min_\phi \log\sigma(\phi)$$
where $\min$ is taken over all differences of markings (it is attained
by Arzela-Ascoli). The following
are the basic properties of $d$ (see e.g. \cite{FM}).

\begin{prop}
\begin{enumerate}[(i)]
\item $d(x,y)\geq 0$ with equality only if $x=y$.
\item $d(x,z)\leq d(x,y)+d(y,z)$ for all $x,y,z\in\X_n$.
\item $d$ is a geodesic metric; for any $x,y$ there is a path from $x$
  to $y$ whose length is $d(x,y)$. Moreover, the path can be taken to
  be piecewise linear, and in fact linear in each simplex.
\item $Out(F_n)$ acts on $\X_n$ by isometries.
\end{enumerate}
\end{prop}

\subsection{Asymmetry and the Main Theorem}

However, in general $d(x,y)\neq d(y,x)$. The following three examples
have motivated our Main Theorem.

\begin{example}
Let $x_k$,
$k\geq 2$, denote $(R_2,id,\ell_k)$ where $\ell_k$ assigns lengths $\frac
{1}{k}$ and $1-\frac{ 1}{k}$ to the two edges of the rose $R_2$. Then
$d(x_2,x_k)<\log 2$ for all $k$ while $d(x_k,x_2)=\log\frac
k2\to\infty$ as $k\to\infty$. Note that in this case the asymmetry can
be explained by the fact that the injectivity radius $injrad(x_k)$ of
$x_k$ goes to 0, and in fact $d(x_k,x_2)\sim -\log\ injrad(x_k)$.
\end{example}

\begin{example}
Let $\Gamma_{\epsilon,t}$ be the graph consisting of two circles of
lengths $\epsilon$ and $1-\epsilon-t$ connected by an arc of length
$t$, where $0<\epsilon<0.1$ and $0\leq t<1-\epsilon$. Then
$d(\Gamma_{\epsilon,0},\Gamma_{\epsilon,1-2\epsilon})=\log(1+t)<\log
2$ while
$d(\Gamma_{\epsilon,1-2\epsilon},\Gamma_{\epsilon,0})=\log\frac{1-\epsilon}\epsilon\to\infty$
as $\epsilon\to 0$. In this example both graphs have the same
injectivity radius, but one graph has two small loops and the other
only one.
\end{example}

\begin{example}
A more subtle example is illustrated in Figure \ref{subtle}.

\begin{figure}[H]
\begin{center}
\input{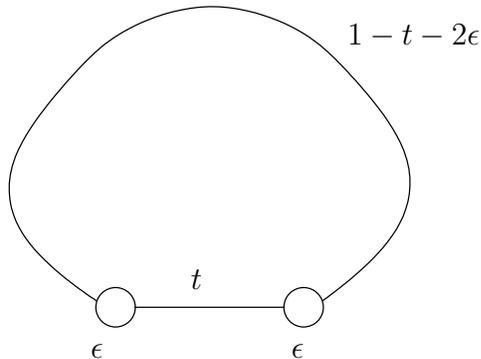}
\caption{\label{subtle} 
The asymmetry of Lipschitz metric is not due to the change in injectivity radius.}
\end{center}
\end{figure}

Fix $0<\epsilon<0.1$ and let $x_t$ denote the graph consisting of two
loops of size $\epsilon$ and two arcs connecting them with lengths $t$
and $1-t-2\epsilon$, with $0\leq t<\frac 13$. Then
$d(x_0,x_t)=\log(1+\frac t{\epsilon})$ while
$d(x_t,x_0)=\log\frac{1-\epsilon}{1- \epsilon -t}$ (see the next section for a calculation
of the distances). Thus $d(x_t,x_0)$ is uniformly bounded, but
$d(x_0,x_t)$ can be made arbitrarily large by choosing suitable
$\epsilon$ and $t$. Here both graphs have the same injectivity radius
and both have two small embedded circles, but one graph has a third
short (non-embedded) loop.
\end{example}

\begin{MRT}
There is an $Out(F_n)$-invariant continuous function $\Psi:\X_n\to\R$ and
constants $A,B>0$ (depending only on $n$) such that for every
$x,y\in\X_n$ we have
$$d(y,x)\leq A\ d(x,y)+B[\Psi(y)-\Psi(x)]$$
\end{MRT}

\subsection{Candidates and computing distances}

We say that a loop $\alpha$ in $\Gamma$ is a {\it candidate} if either
\begin{itemize}
\item it is embedded, or
\item (figure eight) there are two embedded circles $u,v$ in $\Gamma$
  that intersect in one point and $\alpha$ crosses $u,v$ once and does
  not cross any edges outside of $u$ and $v$, or
\item (barbell) there are two disjoint embedded circles $u,v$ in
  $\Gamma$ connected by an arc $w$ whose interior is
  disjoint from $u$ and $v$, such that $\alpha$ crosses $u,v$ once, $w$
  twice, and no edges outside $u\cup v\cup w$.
\end{itemize}

The following basic fact which can be found in \cite{FM} allows us to effectively compute the distance between points. 

\begin{prop}\label{green graph}
Let $x,y\in \X_n$, $x=(\Gamma,f,\ell), y=(\Gamma',f',\ell')$ and let
$\phi:\Gamma\to\Gamma'$ be a difference of markings. Then
there is a candidate loop $\alpha$ in $\Gamma$ such that
$$d(x,y)=\log\frac{\ell'(\phi(\alpha))}{\ell(\alpha)}$$ 
\end{prop}

Note that $d(x,y)\geq \log\frac{\ell'(\phi(\alpha))}{\ell(\alpha)}$
for {\it any} loop $\alpha$. The right hand side does not depend on a
particular choice of $\phi$, so one can effectively compute the
distance by maximizing the ratio over the finitely many candidate
curves. \\

For $x,y \in \Sigma_\Gamma$ one may ask if $d(x,y)$ may be computed
using the same candidate of $x$ when varying $y$ slightly and keeping
$x$ fixed. Proposition \ref{stability} gives a positive answer to this
question under some conditions.

Recall that a {\it closed convex cone} in a finite dimensional real
vector space $V$ is a closed subset $C\subseteq V$ such that $v,w\in C$ implies
$tv+sw\in C$ for all $t,s\in [0,\infty)$. For example, the set of
  integrable vectors in $T_\ell\Sigma_\Gamma$ is a closed convex
  cone. 
  
\emph{Notational convention:} When we restrict our attention to a specific simplex $\Gamma_{(\Sigma, f)}$ in Outer Space we may identify the point $(\Sigma, f, \ell)$ by only specifying $\ell$. 

\begin{prop}\label{stability}
\begin{enumerate}[(i)]
\item\label{itemA} 
Let $\tau\in T_\ell(\Sigma_\Gamma)$ be an integrable vector. Then
there is a candidate loop $\alpha$ in $\Gamma$ such that
$$d(\ell,\ell+t\tau)=\log\frac{(\ell+t\tau)(\alpha)}{\ell(\alpha)}$$
for all sufficiently small $t\geq 0$,
i.e. the same $\alpha$ realizes the distance $d(\ell,\ell+t\tau)$ for
small $t$. Moreover, $\al$ has the property that for any other loop $\beta$, $\frac{\tau(\beta)}{\ell(\beta)} \leq \frac{\tau(\al)}{\ell(\al)}$
\item $\lim_{t\to
  0^+}\frac{d(\ell,\ell+t\tau)}t=\frac{\tau(\alpha)}{\ell(\alpha)}$ for the loop $\al$ in item (\ref{itemA}).
\item The set of integrable vectors in $T_\ell(\Sigma_\Gamma)$ is a
  finite union of closed convex cones $B_1,B_2,\cdots,B_N$ such that
  for any $B_i$ there is a candidate loop $\alpha_i$ that realizes the
  distance $d(\ell,\ell+t\tau)$ for any $\tau\in B_i$ and small $t\geq
  0$.

\end{enumerate}
\end{prop}

\begin{proof}
A candidate $\alpha$ realizes $d(\ell,\ell+t\tau)$ if and only if
$$\frac{(\ell+t\tau)(\alpha)}{\ell(\alpha)}\geq
\frac{(\ell+t\tau)(\beta)}{\ell(\beta)}$$ for all other candidates
$\beta$ in $\Gamma$, 
which simplifies to $\frac{\tau(\alpha)}{\ell(\alpha)}\geq
\frac{\tau(\beta)}{\ell(\beta)}$ when $t>0$.
This is a finite system
of linear inequalities which determines a closed convex cone associated to
$\alpha$ as in (iii). The inequalities do not depend on $t$, proving
(i). Finally (ii) follows from (i) by dividing by $t$ and taking the limit.
\end{proof}

\section{Finsler metric}

\begin{definition}
Let $\tau\in
T_\ell(\Sigma_\Gamma)$. Define
$$\|(\ell,\tau)\|^L=\sup\left\{\left. \frac{\tau(\alpha)}{\ell(\alpha)}\right|
  \alpha\mbox{ is a loop in }\Gamma\right\}$$
\end{definition}

\begin{prop}
\begin{enumerate}[(1)]
\item If $\tau$ is integrable, then
  $\|(\ell,\tau)\|^L=\lim_{t\to 0^+}\frac{d(\ell,\ell+t\tau)}t$. 
\item The supremum in the definition is achieved on a candidate loop
  of $\Gamma$.
\item $\|(\ell,\tau)\|^L$ is continuous on $T(\Sigma)$.
\item $\|(\ell,\tau)\|^L\geq 0$ with equality only if $\tau=0$.
\item $\|(\ell,\tau_1+\tau_2)\|^L\leq
  \|(\ell,\tau_1)\|^L+\|(\ell,\tau_2)\|^L$. 
\item If $c>0$ then $\|(\ell,c\tau)\|^L=c\|(\ell,\tau)\|^L$.

\end{enumerate}
\end{prop}

\begin{proof}
\begin{enumerate}[(1)]
\item
If $\tau$ is integrable, it belongs to the convex cone associated with some  candidate $\alpha$. Proposition \ref{stability}(ii) establishes that
$\frac{\tau(\alpha)}{\ell(\alpha)}=\lim_{t\to 0^+}\frac
{d(\ell,\ell+t\tau)}{t}$, and by Proposition \ref{stability}(i) 
$\frac{\tau(\beta)}{\ell(\beta)}\leq
\frac{\tau(\alpha)}{\ell(\alpha)}$ for any other loop $\beta$. 

\item
When $\tau$ is integrable this follows from (1) and Proposition
\ref{stability}. 
Now let $\tau$ be nonintegrable and
suppose there is a loop $\alpha$ with
$\frac{\tau(\alpha)}{\ell(\alpha)}>\frac{\tau(\beta)}{\ell(\beta)}$
for every candidate $\beta$. Then the same is true after a small
perturbation of $\ell$ where $\tau$ becomes integrable, contradiction. 

\item
This follows from (2), since in the definition we can replace $\sup$
by a maximum over a finite set (of candidates for graphs in $\Sigma$).

\item 
If $\tau\neq 0$ we need to produce some loop $\alpha$ so that
$\tau(\alpha)>0$. This statement does not depend on $\ell$ so we may
assume that $\tau$ is integrable, and then the statement follows from
Proposition \ref{stability}(i): for small enough $t$: $0 < d(\ell, \ell + t \tau) = \log( 1 + t \| (\ell,\tau) \|^L )$.

\item[(5),(6)] This is evident.

\end{enumerate}
\end{proof}

Thus we have an (asymmetric) norm for the Lipschitz metric (homogeneity
(6) only holds for positive scalars). Note first
that the norm is not quasi-symmetric by the example
in Figure \ref{LipNotSym}.

\begin{figure}[H]
\begin{center}
\input{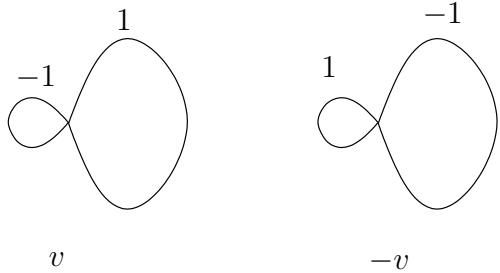}
\caption{\label{LipNotSym}The labels on the edges of $x$ represent the
  vector $\tau$ and $-\tau$. $\|(\ell,\tau)\|^L \sim 1$ and
  $\|(\ell,-\tau)\|^L \sim \frac{1}{\text{length of short loop}} >>
  1$}.
\end{center}
\end{figure}

Next we analyze the relationship between $\|(\ell ,\tau)\|^L$ and $\| (\ell ,
-\tau) \|^L$. The reader may check that when $\tau$ is integrable
$\| (\ell ,-\tau) \|^L=\lim_{t\to 0^+}\frac{d(\ell+t\tau,\ell)}t$.

\section{A corrected Finsler metric}\label{3}

We aim to define a new norm on $T_\ell\Sigma_\Gamma$ which is
quasi-symmetric. The idea is to correct $\|\cdot\|^L$ by adding the
directional derivative of a function that, roughly speaking, is the
sum of $-\log$'s of the lengths of candidates. Since candidates change
from simplex to simplex, we observe that each candidate lifts to an
embedded loop in a suitable double cover and the curves we sum over
are shortest loops in $\Z_2$-homology classes of all double covers.

First consider a nontrivial homology class $a \in H_1(\Gamma;\Z_2)$. By
$\ell(a)$ denote the minimal $\ell(\alpha)$ where $\alpha$ ranges over
loops in the class of $a$. Since there are only finitely many loops of
$\ell$-length bounded above, this minimum exists, but it might be
realized on more than one loop, say $\alpha_1, \dots,\alpha_k$. 

\begin{prop}\label{homo_repn}
For each $a \in H_1(\Gamma, \Z_2)$ there are finitely many loops $\al_1, \dots , \al_k$ so that $\ell(a)$ is realized by some $\al_i$ for all $\ell \in \Sig_\Gamma$. Moreover, if $\al$ is an embedded loop then for all $\ell \in \Sig_\Gamma$, $\al$ is the shortest loop representing $[\al]$.
\end{prop}
\begin{proof}
We claim that if $a \in H_1(\Gamma; \Z_2)$ is represented by $\al$ which realizes $\ell(a)$ and $\al$ 
crosses the edge $e$ more
than once then $\al$ crosses $e$ exactly twice in opposite directions
and $e$ separates the image of $\al$. 
 
To see this we consider two
cases. In the first case, suppose $\al$ crosses $e$ twice in the same
direction. Then up to free homotopy $\al = e \beta_1 e \beta_2$. Construct $\al' = e \beta_1 \overline{\beta_2} \bar{e}$ which is homotopic to $\beta_1
\overline{\beta_2}$. 
$\al', \al$ are $\ZZ_2$ homologous but $\al'$ is strictly shorter than $\al$. In
the second case, suppose $\al$ crosses $e$ twice in opposite
directions and $e$ doesn't separate the image of $\al$. Then $\al = e
\beta_1 \bar{e} \beta_2$ where $\text{Im}\beta_1 \cap \text{Im}\beta_2
\neq \emptyset$.  Let $p \in \text{Im}\beta_1 \cap
\text{Im}\beta_2$. Then we can
write $\al = \gamma_1 e \gamma_2 \gamma_3 \bar{e} \gamma_4$ where $p =
i(\gamma_1) = t(\gamma_2) = i(\gamma_3) = t(\gamma_4)$ (here we use $i(\cdot)$ for the initial point and $t(\cdot)$ for the terminal point). We also have
$t(\gamma_1) = i(\gamma_4)$ and $i(\gamma_2) = t(\gamma_3)$. Construct
$\al' = \gamma_1 e \bar{e} \gamma_4 \gamma_3 \gamma_2 \sim \gamma_1
\gamma_4 \gamma_3 \gamma_2 $. $\al'$ also represents $a$ but it has
strictly shorter length. 

We conclude that if $\al$ is
the shortest loop representing $a$ then for each edge $e$ in its
image, $\al$ either crosses $e$ once or it crosses $e$ twice in
opposite directions and $e$ separates the image of $\al$. For each $a$
there are only finitely many such loops $\al$ (and they don't depend
on $\ell$).

 For the second part, it is elementary to see that if $\al$ is embedded and
$\beta$ is another loop with $\beta$ homologous mod $\ZZ_2$ to $\al$ then $\beta$ crosses
all the edges of $\al$. Thus $\al$ is a shortest loop representing its homology class, and any other loop representing the same homology class with the same length must be a reparametrization of $\al$.
\end{proof}

The set of linear inequalities $\ell(\al_i) \leq l(\al_j)$ for the set
of $\al_i$s in Proposition \ref{homo_repn} divides the simplex
$\Sigma_\Gamma$ into closed convex subsets $C_1, \dots , C_k$ such
that for each $C_i$ there is an $\al_j$ so that $\ell(\al_i) \leq
\ell(\al_j)$ for all $j$.

\begin{cor}\label{divideSigCor}
A simplex $\Sig_\Gamma$ is covered by closed convex subsets $C_1,
\dots , C_k$ so that for each $a \in H_1(\Gamma, \ZZ_2)$ there is a
loop $\al_j$ such that $\ell(a) = \ell(\al_j)$ for all $\ell \in C_j$.
\end{cor}

\begin{cor}\label{derivative}
When $\tau\in T_\ell\Sigma_\Gamma$ is integrable there is a $j$ such
that $\ell, \ell + t\tau \in C_j$ (for all small $t>0$) and the
derivative from the right at 0 of $t\mapsto (\ell+t\tau)(a)$ is
$\tau(\alpha_j)$. In other words, it equals
$$\max\{\tau(\alpha)\mid \alpha\mbox{ realizes }\ell(a)\}$$
\end{cor}

Let $\Gamma_i\to\Gamma$, $i=1,2,\cdots,2^n-1$ be the collection
of all nontrivial double covers of $\Gamma$. Any
$\ell\in\Sigma_\Gamma$ induces a metric $\ell_i$ on each $\Gamma_i$ by
pulling back, and likewise any tangent vector $\tau\in
T_\ell\Sigma_\Gamma$ lifts to a tangent vector in
$T_{\ell_i}\Sigma_{\Gamma_i}$. If $a \in H_1(\Gamma_i;\Z_2)$ is a given
homology class, denote by $\ell_i(a)$ the length of a shortest loop in
$\Gamma_i$ equipped with $\ell_i$ that represents $a$.

\begin{lemma}\label{canAreRep}
If $\al$ is a candidate in $\Gamma$ then there exists a double cover
$\Gamma_i\to\Gamma$, and a lift $\tilde \al$ of $\al$ so that $\tilde
\al$ is the unique shortest loop in its (nontrivial) homology class.
\end{lemma}

\begin{proof}
We will show that we can arrange that $\tilde\al$ is embedded, and
this will guarantee that $\tilde\al$ is shortest in its homology
class. If $\alpha$ is embedded, then any double cover to which
$\alpha$ lifts works. If
$\alpha$ is a figure eight or a barbell, take the double cover by
cutting and regluing along two points, one in each embedded loop of
$\alpha$ (so $\alpha$ lifts but the embedded loops don't).
\end{proof}

\begin{figure}[H]
\begin{center}
\input{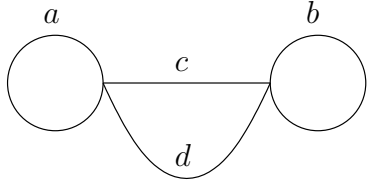}
\caption{\label{canNotRepn} The homology class of $acb\bar c$ and $adb
  \bar d$ are equal. If $c$ is shorter than $d$ then $adb\bar d$ will
  not be a homology representative. However it is the image of a homology representative in some double cover of this graph.}
\end{center}
\end{figure}

Now we may define the new norm,

\begin{defn}
Let 

\begin{equation}\label{bigsum}
N(\ell,\tau)=
-\sum_{\Gamma_i}\sum_{ a \in
  H_1(\Gamma_i;\Z_2)\setminus\{0\}}\frac{\max\tau(\alpha)}{\ell(a)}
\end{equation}

where maximum is taken over all loops $\alpha$ in $\Gamma_i$ that realize
$\ell(a)$. Note that some of the terms in the sum may be negative
(e.g. generically, there is only one $\alpha$ realizing $\ell(a)$).

Define the new norm by

\begin{equation}\label{newNorm1}
 \| (\ell,\tau) \|^N = \| (\ell,\tau) \|^L + \frac{1}{K+1} N(\ell,\tau) \\[0.3 cm]
 \end{equation}

where $K=(2^n-1)(2^{2n-1}-1)$ is the number of summands in (\ref{bigsum}).
\end{defn}

When $\ell$ can be understood from the context, we will sometimes
write $\| \tau \|^\cdot $ instead of $\| (\ell,\tau) \|^\cdot$ for simplicity.

\begin{lemma}\label{newIsMax}
\[ \frac 1{K+1}\max\{\|\tau\|^L,\|-\tau\|^L\}\leq 
\| \tau \|^N \leq 2\|\tau\|^L+\|-\tau\|^L \]
\end{lemma}

\begin{proof}
Let $\gamma$ be a loop realizing $\| \tau \|^L$, and $\al$ a loop
realizing $\| - \tau \|^L$. Recall that for all loops $\beta$ in
$\Gamma$, $\frac{\tau(\beta)}{\ell(\beta)} \leq
\frac{\tau(\gamma)}{\ell(\gamma)} = \| \tau \|^L$ and
$\frac{-\tau(\beta)}{\ell(\beta)} \leq \frac{-\tau(\al)}{\ell(\al)} =
\| -\tau \|^L$. The max in (\ref{bigsum}) goes over loops $\beta$ such that $\ell(\beta) = \ell(a)$ thus for each $\beta$ in the sum  $\frac{\tau(\beta)}{\ell(a)} \leq \| \tau \|^L$ and $\frac{-\tau(\beta)}{\ell(a)} \leq 
\| -\tau \|^L$. Hence the right inequality in the statement follows.

Inequality $\frac 1{K+1}\|\tau\|^L \leq \|\tau\|^N$
is equivalent to $-N(\ell,\tau)\leq K\|\tau\|^L$ which is
again evident, since the positive summands on the left hand side are
dominated by $\|\tau\|^L$.

Finally, inequality $\frac 1{K+1}\|-\tau\|^L \leq
\|\tau\|^N$ can be rewritten as
$$\|-\tau\|^L-N(\ell,\tau)\leq (K+1)\|\tau\|^L$$
All positive terms in $-N(\ell,\tau)$ are dominated by
$\|\tau\|^L$ as before. If $\alpha$ is a candidate that realizes
$\|-\tau\|^L$ then there is a term in $-N(\ell,\tau)$ of the form
$\frac{\tau(\alpha)}{\ell(\alpha)}$ that cancels
$\|-\tau\|^L=\frac{-\tau(\alpha)}{\ell(\alpha)}$.
\end{proof}

Thus $\|\cdot\|^N$ is a (non-symmetric) norm, just like $\|\cdot\|^L$
(positivity follows from Lemma \ref{newIsMax} and subadditivity is
evident from the definition).
The next corollary states that the new norm, unlike
$\|\cdot\|^L$,  is quasi-symmetric. 

\begin{cor}\label{NewQuasiSym}
There is a constant $A=3(K+1)$ so that 
\[ \| \tau\|^N \leq A \phantom{,} \| -\tau\|^N  \]
\end{cor}

Define the map $\Psi: \Sigma_\Gamma \to \RR$ by 
\begin{equation}\label{Psieq}
\Psi(\ell) = - \frac{1}{K+1} \sum_{\Gamma_i} \sum_{a \in H_1(\Gamma_i;\Z_2)\setminus\{0\}} \log \ell_i(a)
\end{equation}
where $\ell_i$ is the lift of $\ell$ to $\Gamma_i$. Note that $\Psi$
is smooth (even analytic) on each convex set $C_j$ of Corollary
\ref{divideSigCor}.

\begin{prop}\label{newVsLip}
If $\ell \in \Sigma_\Gamma$ and $\tau \in T_\ell\Sigma_\Gamma$ is
integrable then
\[ \| \tau \|^N = \| \tau \|^L + d_\tau\Psi \]
where the third term is the derivative of $\Psi$ in the direction of
$\tau$, i.e. the derivative from the right at 0 of $t\mapsto
\Psi(\ell+t\tau)$. 
\end{prop}

\begin{proof}
Applying Corollary \ref{derivative} to $\Gamma_i,\ell_i$ and
the lift $\tau_i$ of $\tau$ to $\Gamma_i$, we obtain that 
$$d_{\tau_i}\ell_i(a)=\tau_i(\alpha_i)$$ where $\alpha_i$ is a curve
that realizes $\ell_i(a)$ on which $\tau_i$ is maximal. Thus
$$d_{\tau_i}\log\ell_i(a)=\frac{\tau(\alpha_i)}{\ell_i(\alpha_i)}$$
and adding gives $d_\tau\Psi= \frac 1{K+1}N(\ell,\tau)$, and the claim
follows. 
\end{proof}

We can easily extend this discussion to the whole Outer space
$\X_n$. It is easy to see that $\|\cdot\|^L,\|\cdot\|^N$ and $\Psi$
commute with inclusions of simplices corresponding to collapsing
forests. If $f:R_n\to \Gamma$ is a marking, $f_*:H_1(R_n;\Z_2)\to
H_1(\Gamma;\Z_2)$ is an isomorphism and we identify homology classes
in $H_1(\Gamma;\Z_2)$ with homology classes in
$H_1(R_n;\Z_2)$. Similarly, cohomology can be identified, i.e. the
double covers of $\Gamma$ with double covers of $R_n$, and $f$ lifts
to markings of double covers of $\Gamma$ by double covers of
$R_n$. This means that $\Psi:\X_n\to\R$ can be defined
globally. Moreover, changing the marking only permutes the summands in
the definition of $\Psi$, so $\Psi$ is $Out(F_n)$-invariant. 

\section{Lengths of paths}

Let $\path: [0,1] \to \X_n$ be a piecewise linear path. In particular,
$\path$ can be subdivided into finitely many subpaths so that each is
contained in one of the convex sets of Corollary \ref{divideSigCor} on
which $\Psi$ is smooth.  Then the Lipschitz length of $\path$ is
\[ len_L\path = \sup \left\{ \sum_{i=1}^pd(\path(t_{i-1}), \path(t_i))
\mid 
0=t_0 < t_1 < \dots < t_p=1 \right\} \] 
Suppose $\Delta t_i = t_i-t_{i-1}$ is small. Then
\[ d(\path(t_{i-1}),\path(t_i)) = 
\frac{d(\path(t_{i-1}),\path(t_{i-1}+\Delta t_i))}{\Delta t_i} 
\cdot \Delta t_i \sim \| (\path(t_{i-1}), \dot{\path}(t_{i-1})) \|^L 
\Delta t_i \] 
Thus 
\[ len_L\path = \int_0^1 \| (\path(t), \dot{\path}(t)) \|^L dt\]
Define the new length of the path $\path$ 
\[ len_N\path = \int_0^1 \| (\path(t), \dot{\path}(t)) \|^N dt\]

\begin{prop}\label{newVsLipForPaths}
Let $\path:[0,1] \to \X_n$ be a path from $x$ to $y$ in $\X_n$. Then 
\[ len_N(\path) = len_L(\path) +  \Psi(y) -  \Psi(x) \]
\end{prop}

\begin{proof}
Since $\Psi$ is piecewise differentiable and $\path$ is piecewise
linear we may apply the Fundamental Theorem of Calculus to $\Psi \circ
\path$. Thus, by Proposition \ref{newVsLip}
\[  \begin{array}{lll}
	len_N(\path) & = & \int_0^1 \| \dot\path(t) \|^N dt \\[0.3 cm ]
	
	& = & \int_0^1 \left[ \| \dot\path(t) \|^L +
          d_{\dot\path(t)}\Psi \right] dt \\[0.3 cm ] & = &
        len_L(\path) + \Psi(y) - \Psi(x)
\end{array} \]
\end{proof}

\begin{prop}\label{newIsSymm}
Let $\path:[0,1] \to \X_n$ be a path from $x$ to $y$. Let
$-\path:[0,1] \to \X_n$ be the reverse path $-\path(t) =
\path(1-t)$. Then \[ len_N(-\path) \leq A \phantom{,} len_N(\path) \]
where $A$ is the constant from Corollary \ref{NewQuasiSym}.
\end{prop}

\begin{proof}
Since $\path$ is piecewise $C^1$, for all but finitely many points 
$\dot{[-\path]}(s) = - \dot\path(1-s)$. Thus
\[  \begin{array}{lll}
len_N(-\path) & = &  \int_0^1  \| \dot{[-\path]}(s) \|^N ds  \\[0.3 cm]
& = & \int_1^0 \| - \dot\path(t) \|^N (-dt) \\[0.3 cm ]
& = & \int_0^1 \| - \dot\path(t) \|^N dt \leq  \int_0^1 A \phantom{,} \| \dot\path(t) \|^N dt \\ [0.3 cm] 
& = & A \phantom{a} len_N(\path)
\end{array} \]
\end{proof}

\section{Applications}

We now elevate the local results obtained in section \ref{3} to global
results on the lengths of paths and distances. Let $A$ be the constant
from Corollary \ref{NewQuasiSym}.
\begin{cor}
For any $\phi \in \out$ and any piecewise linear path $\path$ from $x$
to $x \cdot \phi$, \[ len_L(\path) = len_N(\path)\] Therefore \[len_L(\path)
\leq A \phantom{a} len_L(-\path)\] 
\end{cor}
\begin{proof}
By Proposition \ref{newVsLipForPaths}, $len_N(\path) = len_L(\path) +  \Psi(x \cdot \phi) - \Psi(x) $. But since $\Psi(x) = \Psi(x \cdot \phi)$ we get $len_N(\path) = len_L(\path)$.
\end{proof}

\begin{theorem}\label{lipLen}
For any piecewise linear path $\path$ from $x$ to $y$ 
\[ len_L (\path) \leq A \phantom{,} len_L(-\path) + (A+1) [ \Psi(x) - \Psi(y) ]\]
\end{theorem}

\begin{proof}
By Propositions \ref{newVsLipForPaths} and \ref{newIsSymm}:  
$len_L(\path) + \Psi(y) - \Psi(x) = len_N(\path) \leq A \phantom{,} len_N(-\path) =  A \phantom{,} len_L(-\path) + A \left[ \Psi(x) - \Psi(y) \right]$.
\end{proof}

\begin{MRT}
For any $x,y \in \X_n$
\[ d(x,y) \leq A \cdot d(y,x) + (A+1) \left[ \Psi(x) - \Psi(y) \right]  \]
\end{MRT}

\begin{proof}
Apply Theorem \ref{lipLen} to $\path$, where $-\path$ is a geodesic from $y$ to $x$. 
\end{proof}

In particular, since $\Psi$ is $\out$ invariant, if $x,y \in \X_n$ are in the same orbit then $d(x,y) \leq A d(y,x)$. 

\begin{remark}
The theorem above is equivalent to
$$ \max\{d(x,y),d(y,x) \} \asymp \max\{ \phantom{,} \min\{ d(x,y), d(y,x) \}, \phantom{,} | \Psi(y) - \Psi(x)| \phantom{,} \}$$
This follows from
 
\begin{claim}
If $d(x,y) \geq 2A \phantom{,} d(y,x)$ then $$\frac{d(x,y)}{2(A+1)} \leq  \Psi(x) - \Psi(y) \leq d(x,y)$$ 
\end{claim}
\begin{proof}
From $0 \leq d(y,x) \leq A d(x,y) + (A+1)[\Psi(y) - \Psi(x)]$ we get
that $\Psi(x) - \Psi(y) \leq d(x,y)$. From $d(x,y) \leq A d(y,x) +
(A+1)[\Psi(x) - \Psi(y)]$ we get that $\Psi(x) - \Psi(y) \geq
\frac{1}{A+1} \left( d(x,y) - A d(y,x) \right) \geq
\frac{d(x,y)}{2(A+1)}$
\end{proof}
\end{remark}

The following theorem is due to Handel-Mosher \cite{HM}.

\begin{theorem}\label{HMTheorem}
For any irreducible automorphism $\Phi \in \out$, let $\lam$ be
the expansion factor of $\Phi$ and $\mu$ the expansion factor of
$\Phi^{-1}$. Then $\mu \leq \lam^A$.
\end{theorem}

\begin{proof}
Let $f:\Gamma \to \Gamma$ be a train track representative of $\Phi$
and $g: \Gamma' \to \Gamma'$ a train track representative for
$\Phi^{-1}$. Let $D \geq d(\Gamma,\Gamma'), d(\Gamma',\Gamma)$. Then
$d(\Gamma \cdot \Phi^j,\Gamma) \leq A \cdot d(\Gamma,\Gamma \cdot
\Phi^j) = A \cdot j \cdot \log \lam$. On the other hand, $d(\Gamma \cdot \Phi^j, \Gamma) \geq d(\Gamma'
\cdot \Phi^j, \Gamma') - d(\Gamma' \cdot \Phi^j, \Gamma \cdot \Phi^j)
- d(\Gamma,\Gamma') \geq j \log\mu - 2D$. Therefore
\[ A j \log\lam \geq j \log \mu - 2D \] Thus for every $j$, 
$\log\mu \leq A \log\lam + \frac{2D}{j}$ which implies 
$\frac{\log\mu}{\log\lam} \leq A$.
\end{proof}

Let $\X_{\geq \eps}$ be the set of all marked metric graphs in $\X_n$
which don't contain loops shorter than $\eps$. 

\begin{theorem}\label{sym_in_thick}
For every $\eps>0$ there is a constant $B$ so that for any $x,y \in
\X_{\geq \eps}$ and any piecewise linear path $\path$ from $x$ to $y$:
\[ \frac{1}{A} \phantom{,}  len(\path) - B \leq len(-\path) \leq A 
\phantom{,} len(\path) + B\] 
Moreover, there is a constant $D$ such that for all $x,y\in \X_{\geq\eps}$
\[ d(y,x) \leq D  \phantom{,} d(x,y) \]
\end{theorem}

\begin{proof}
Since $\Psi$ is continuous and $\X_{\geq \eps}/\out$ is compact, there is
a $C = K\log \frac{1}{\eps}$ so that for every $x \in \X_{\geq \eps}$: $|\Psi(x)| \leq
C$. Then by Propositions \ref{newVsLipForPaths} and \ref{newIsSymm},
\[\begin{array}{lll}
len_L(\path) - 2C & \leq & len_L(\path) + \Psi(y) -\Psi(x) = 
len_N(\path) \leq A \phantom{,} len_N(-\path) \\[0.3 cm]
& = & A( len_L(-\path) + \Psi(x) - \Psi(y) ) \leq 
A \phantom{,} len_L(-\path)  + 2AC
\end{array} \]
From the Main Theorem we see that $d(y,x)\leq A d(x,y)+B$ for any
$x,y\in \X_{\geq \eps}$. We now need to remove the additive
constant. If $d(x,y)\geq \log 2$ the additive constant can be absorbed
in the multiplicative constant: $d(y,x)\leq A d(x,y)+B\leq
(A+B/\log 2)d(x,y)$. So suppose $d(x,y)\leq \log 2$. Let
$\path$ be a geodesic path from $x$ to $y$. Then $\path$ must stay
inside $\X_{\geq\eps/2}$. By cocompactness, there is some $M$ so that
$\|-\tau\|^L\leq M\|\tau\|^L$ for all tangent vectors $\tau$ based at
a point in $\X_{\geq\eps/2}$. Thus $d(y,x)\leq len_L(-\path)\leq M
len_L(\path)=M d(x,y)$.
\end{proof}

\bibliography{./ref}

 \end{document}